\documentclass[reqno]{amsart}

\usepackage{graphicx,xcolor}
\usepackage{amssymb,amsmath,stmaryrd,txfonts,mathrsfs,amsthm}
\usepackage[foot]{amsaddr}
\usepackage{tikz}

\definecolor{myurlcolor}{rgb}{0,0,0.7}
\usepackage[bookmarks=false]{hyperref}
\hypersetup{colorlinks,
linkcolor=myurlcolor,
citecolor=myurlcolor,
urlcolor=myurlcolor}

\input xy
\xyoption{all}

\newcommand{\maps}{\colon}    
                              
\newcommand{\im}{\mathrm{im\,}} 

\newcommand{\R}{{\mathbb R}}     
\newcommand{\Z}{{\mathbb Z}}     

\newcommand{\A}{{\rm A}}
\newcommand{\B}{{\rm B}}

\newcommand{\F}{{\rm F}}

\renewcommand{\H}{{\rm H}} 
\newcommand{\I}{{\rm I}} 

\renewcommand{\S}{{\rm S}} 
\newcommand{\PGL}{{\rm PGL}}
\newcommand{\Aut}{{\rm Aut}}
\newcommand{\Cov}{{\rm Cov}}

\newcommand{\supp}{\mathrm{supp}}       

\newcommand{\define}[1]{{\bf \boldmath{#1}}}

\newtheorem{thm}{Theorem}    

\newtheorem{lem}[thm]{Lemma}

\newtheorem{defn}[thm]{Definition}

        \newcommand{\be}{\begin{equation}}
        \newcommand{\ee}{\end{equation}}
        \newcommand{\ba}{\begin{eqnarray}}
        \newcommand{\ea}{\end{eqnarray}}
        \newcommand{\ban}{\begin{eqnarray*}}
        \newcommand{\ean}{\end{eqnarray*}}
        \newcommand{\barr}{\begin{array}}
        \newcommand{\earr}{\end{array}}

\title{Topological Crystals}

\author{John\ C.\ Baez$^{1}$}
\address{$^1$School of Mathematics, University of Edinburgh, James Clerk Maxwell Building, Peter Guthrie Tait Road, Edinburgh, UK EH9 3FD.}
\email{baez@math.ucr.edu}

\date{\today} 

\begin{document}
\maketitle

\begin{abstract}
\noindent 
Sunada's work on crystallography emphasizes the role of the `maximal abelian 
cover' of a graph \(X\).  This is a covering space of \(X\) for which the group of deck 
transformations is the first homology group \(H_1(X,\Z)\).  An embedding of the
maximal abelian cover in a vector space can serve as the pattern for a crystal: 
atoms are located at the vertices, while bonds lie along the edges.  We prove that for any 
connected graph \(X\) without bridges, there is a canonical embedding of the maximal 
abelian cover of \(X\)  into the vector space \(H_1(X,\R)\), called a `topological 
crystal'.  Crystals of graphene and diamond are examples of this construction.  We prove 
that any symmetry of a graph lifts to a symmetry of its topological crystal.   We also 
compute the density of atoms in a topological crystal.   The key technical tools are a way of 
decomposing the 1-chain coming from a path in \(X\) into manageable pieces, and the 
work of Bacher, de la Harpe and Nagnibeda on integral cycles and integral cuts.
\end{abstract}

\section{Introduction} \label{sec:intro}

The `maximal abelian cover' of a graph plays a key role in Sunada's work on topological crystallography \cite{SunadaBook}.  Just as the universal cover of a connected graph \(X\) has the fundamental group \(\pi_1(X)\) as its group of deck transformations, the maximal abelian cover, denoted \(\overline{X}\), has the abelianization of \(\pi_1(X)\) as its group of deck transformations.  It thus covers every other connected cover of \(X\) whose group of deck transformations is abelian.  Since the abelianization of \(\pi_1(X)\) is the first homology group \(H_1(X,\Z)\), there is a close connection between the maximal abelian cover and homology theory. 

In this paper we prove that for a large class of graphs, the maximal abelian cover \(\overline{X}\) can naturally be embedded in the vector space \(H_1(X,\R)\).  Following Sunada, we call this embedded copy of \(\overline{X}\) in  \(H_1(X,\R)\) a `topological crystal'.   The symmetries of the original graph can be lifted to symmetries of its topological crystal, but the topological crystal also has an \(n\)-dimensional lattice of translational symmetries.    In the 3-dimensional case, the topological crystal can serve as the blueprint for an actual crystal, with atoms at the vertices and bonds along the edges.  

The most famous example arises from the tetrahedron.   Here \(X\) is the complete graph on 4 vertices:

\begin{center} 
\raisebox{-0.0 em}{\includegraphics[scale = 0.15]{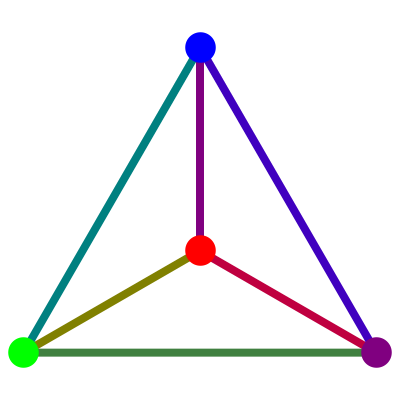}} 
\end{center}

\vskip 0.5em
\noindent 
Since $H_1(X,\R)$ is 3-dimensional, the corresponding topological crystal is a graph embedded in 3-dimensional Euclidean space.  It looks like this:

\begin{center} 
\raisebox{-0.0 em}{\includegraphics[scale = 0.35]{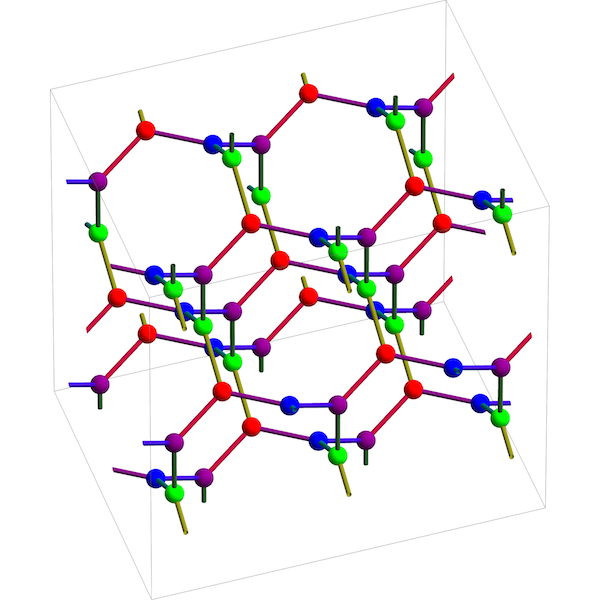}} 
\end{center}

\noindent
This crystal has 4 translation equivalence classes of atoms, corresponding to the 4  vertices of the tetrahedron.  Each atom has 3 equally distant neighbors lying in a plane at \(120{}^\circ\) angles from each other. These planes lie in 4 families, each parallel to one face of a regular tetrahedron. This structure was discovered by the 
crystallographer Laves \cite{HL,Laves} and dubbed the `Laves graph' by  Coxeter \cite{Coxeter}.   Later Sunada called it the `\(\mathrm{K}_4\) lattice', studied its energy minimization properties, and raised the question of whether it could serve as the pattern for a crystal form of carbon \cite{SunadaAMS}.   This form of carbon has not yet been seen, but the Laves graph plays a role in the structure of certain butterfly wings  \cite{HoKP}.

The Laves graph is exciting because it was studied mathematically before being found in nature.   A more familiar example arises from this graph:

\begin{center} 
\raisebox{-0.0 em}{\includegraphics[scale = 0.15]{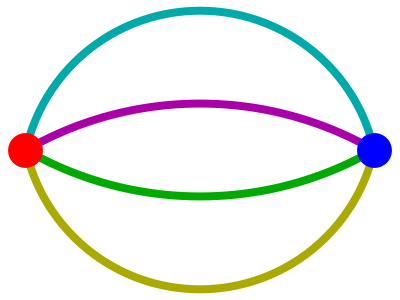}} 
\end{center}

\noindent
The corresponding topological crystal provides the pattern for a diamond:

\begin{center} 
\raisebox{-0.0 em}{\includegraphics[scale = 0.35]{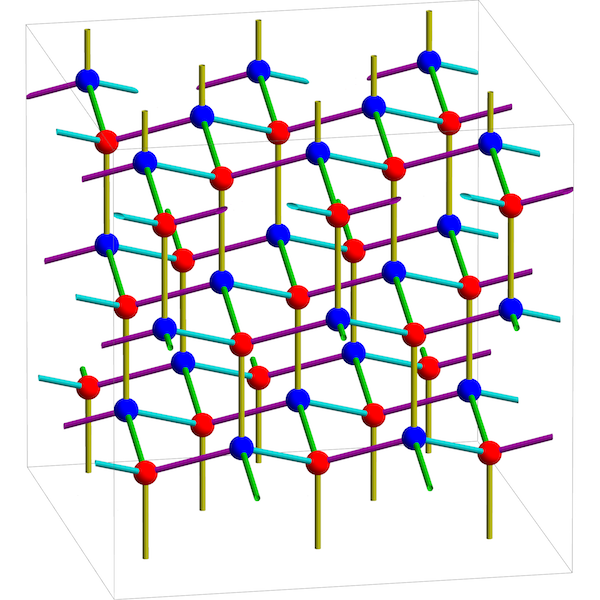}} 
\end{center}

\noindent
This crystal has two translation equivalence classes of atoms, corresponding to the two vertices of the original graph.

The general construction of topological crystals was developed by Eon \cite{Eon1,Eon2}, who used the term `archetypical representation', and later Kotani and Sunada \cite{KotaniSunada}, who used the term `standard realization'.  Sunada uses `topological crystal' for an even more general concept \cite{SunadaBook}, but we only need the following special case.

We start with a graph \(X\).  This has a space \(C_0(X,\R)\) of 0-chains, which are formal linear combinations of vertices, and a space \(C_1(X,\R)\)  of 1-chains, which are formal linear combinations of edges.  There is a boundary operator
\[         \partial \maps C_1(X,\R) \to C_0(X,\R) ,\]
the linear operator sending any edge to the formal difference of its two endpoints.  The kernel of this operator is the space of 1-cycles, \(Z_1(X,\R)\).  There is a unique inner product on the space of 1-chains such that edges form an orthonormal basis.  This determines an orthogonal projection 
\[           \pi \maps C_1(X,\R) \to Z_1(X,\R)  .\]
The vector space \(Z_1(X,\R)\) is isomorphic to the first homology of \(X\) with real coefficients, \(H_1(X,\R)\).  The so-called `topological crystal' of \(X\) is an embedding of its maximal abelian cover \(\overline{X}\) in \(Z_1(X,\R)\).   We obtain this by first embedding \(\overline{X}\) in \(C_1(X,\R)\) and then projecting it down via \(\pi\). 

To accomplish this, we need to fix a basepoint for \(X\).  Each path \(\gamma\) in \(X\) starting at this basepoint determines a 1-chain \(c_\gamma\).   It is easy to show that these 1-chains correspond to the vertices of \(\overline{X}\).  Furthermore, the graph \(\overline{X}\) has an edge from \(c_\gamma\) to \(c_{\gamma'}\) whenever the path \(\gamma'\) is obtained by adding an extra edge to \(\gamma\).   We can think of this edge as a straight line segment from \(c_\gamma\) to \(c_{\gamma'}\).

Our main result, Theorem \ref{embedding_theorem}, says when the projection \(\pi\) maps this copy of \(\overline{X}\) to \(Z_1(X,\R)\) in a one-to-one manner.    Eon \cite{Eon1} and Kotani--Sunada \cite{KotaniSunada} argued that \(\pi\) is one-to-one on the vertices of \(\overline{X}\) if and only if the graph \(X\) has no `bridges': that is, edges whose removal would disconnect \(X\).    We prove this in Theorem \ref{atom_theorem}.   But we go further and show that if and only if \(X\) has no bridges, \(\pi\) maps the whole graph \(\overline{X}\)---not just its vertices but also its edges---into \(Z_1(X,\R)\) in a one-to-one way.   In proving this result, our main technical tool is Lemma \ref{decomposition_lemma_2}, which for any path \(\gamma\) decomposes the 1-chain \(c_\gamma\) into manageable pieces.  

We call the resulting copy of \(\overline{X}\) embedded in \(Z_1(X,\R)\) a `topological crystal'.  For example, let \(X\) be this graph:

\begin{center} 
\raisebox{-0.0 em}{\includegraphics[scale = 0.15]{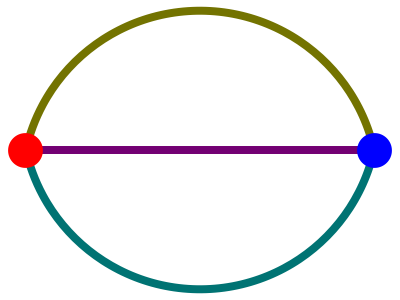}} 
\end{center}

\noindent
Since \(X\) has 3 edges, the space of 1-chains is 3-dimensional.  Since \(H_1(X,\R)\) is 2-dimensional, the space of 1-cycles is a plane in this 3-dimensional space.  If we consider paths \(\gamma\) in \(X\) starting at the red vertex, form the 1-chains \(c_\gamma\), and project them down to this plane, we obtain the following picture:

\begin{center} 
\raisebox{-0.0 em}{\includegraphics[scale = 0.35]{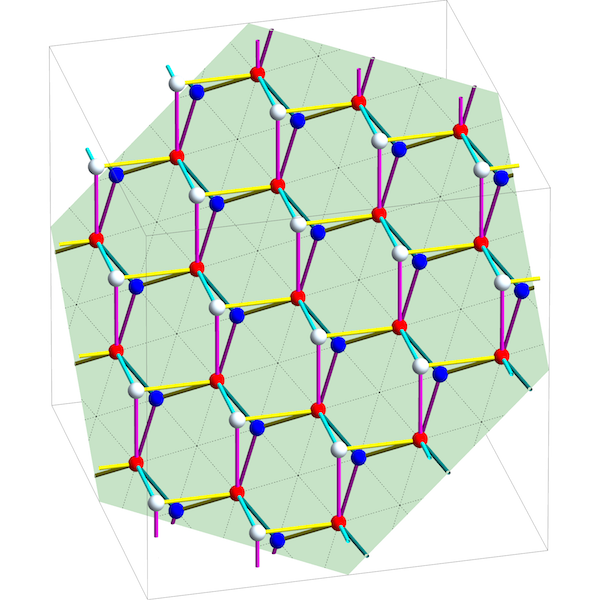}} 
\end{center}

\noindent Here the 1-chains \(c_\gamma\) are the white and red dots.  These are the vertices of \(\overline{X}\), while the line segments between them are the edges of \(\overline{X}\).  Projecting these vertices and edges onto the plane of 1-cycles, we obtain the topological crystal for \(X\).   The blue dots come from projecting the white dots onto the plane of 1-cycles, while the red dots already lie on this plane.  The resulting topological crystal provides the pattern for graphene, a 2-dimensional form of carbon:

\vskip 1em
\begin{center} 
\raisebox{-0.0 em}{\includegraphics[scale = 0.2]{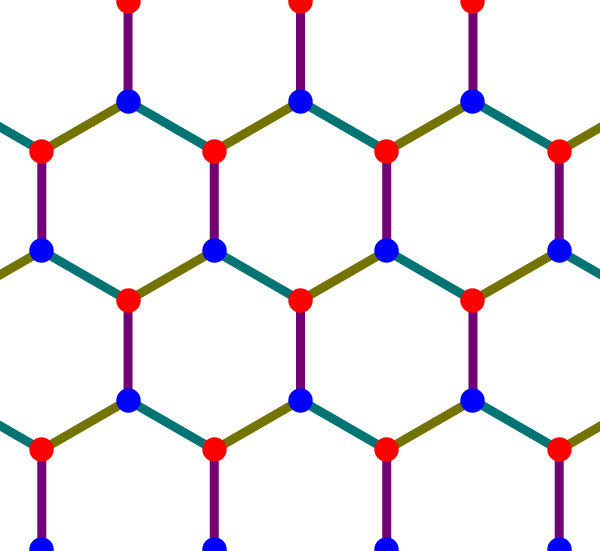}} 
\end{center}

For any connected graph \(X\), there is a covering map
\[        q \maps \overline{X} \to X .\]
The vertices of \(\overline{X}\) come in different kinds, or `colors', depending on which vertex of \(X\) they map to.   We study the group of `covering symmetries', \(\Cov(X)\), consisting of all symmetries of \(\overline{X}\) that map vertices of the same color to vertices of the same color.  In Lemma \ref{covering_lemma_1} we prove that this group fits into a short
exact sequence
\[       1 \longrightarrow H_1(X,\Z) \longrightarrow \Cov(X) \longrightarrow \Aut(X) \longrightarrow 1 \]
where \(\Aut(X)\) is the group of all symmetries of \(X\).  Thus, 
every symmetry of \(X\) is covered by some symmetry of \(\overline{X}\), 
while \(H_1(X,\Z)\) acts on \(\overline{X}\) in a way that preserves the color of every vertex.  In Theorem \ref{symmetry_theorem_2} we prove that \(\Cov(X)\) also acts as symmetries of the topological crystal associated to \(X\). More precisely, there is an action of \(\Cov(X)\) as affine isometries of \(H_1(X,\R)\) for which the embedding of \(\overline{X}\) in this space is equivariant.    In this action, the subgroup \(H_1(X,\Z)\) acts as translations.  In Section \ref{sec:examples} we discuss some examples where \(X\) is highly symmetrical.

Another interesting property of a topological crystal is its `packing fraction'.  The set \(A\) of atoms is contained in the lattice \(L\) obtained by projecting the integral 1-chains down to the space of 1-cycles:
\[      L = \{   \pi(c) : \; c \in C_1(X,\Z)  \}. \]
We can ask what fraction of the points in this lattice are actually atoms.   We call this the `packing fraction', and since \(Z_1(X,\Z)\) acts as translations on both \(A\) and \(L\), we define it to be
\[           \frac{|A/Z_1(X,\Z)|}{|L/Z_1(X,\Z)|}. \]
For example, suppose \(X\) is this graph:

\vskip 1em
\begin{center} 
\raisebox{-0.0 em}{\includegraphics[scale = 0.15]{grapheneGraph.png}} 
\end{center}

\noindent
Then the packing fraction is \(\frac{2}{3}\), as can be seen here:

\vskip 1em
\begin{center} 
\raisebox{-0.0 em}{\includegraphics[scale = 0.20]{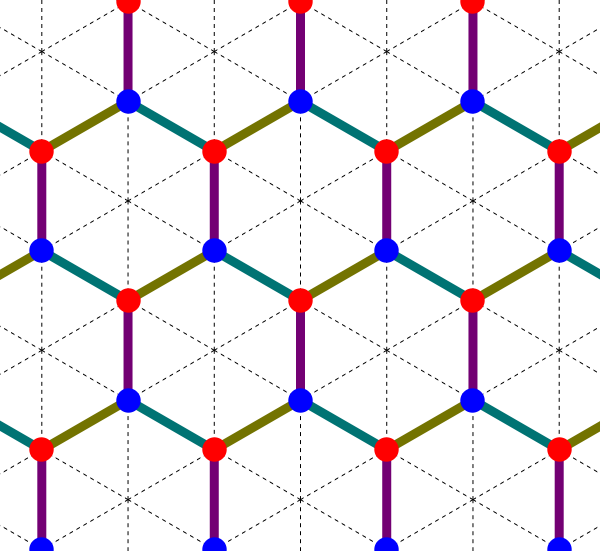}} 
\end{center}
\vskip 1em

\noindent
For any bridgeless connected graph \(X\), we prove in Theorem \ref{packing_fraction_theorem} that
\[           \frac{|A/Z_1(X,\Z)|}{|L/Z_1(X,\Z)|} = \frac{|V|}{|T|}  \]
where \(V\) is the set of vertices and \(T\) is the set of spanning trees.  The main tool here is Bacher, de la Harpe and Nagnibeda's work on integral cycles and integral cuts \cite{BHN}.

\subsection*{Acknowledgments}  

I thank Greg Egan, without whose help this project could never have been done.  I thank David Speyer and Ilya Bogdanov for their help in proving Lemma \ref{bridge_lemma}.  In particular, Speyer found a proof of Lemma \ref{bridge_lemma} based on the lattice of integral cuts \cite{Speyer}, while
Bogdanov noticed the importance of the `no bridges' condition and laid the groundwork 
for the proof here \cite{Bogdanov}.  I thank Pbroks13 and Tom Ruen for the picture of the hexagonal hosohedron, which is available \href{https://commons.wikimedia.org/wiki/File:Hexagonal_Hosohedron.svg}{on Wikicommons} under a Creative Commons Attribution-Share Alike 3.0 Unported license.  I thank Leshabirukov for the picture of the Petersen graph, which is available \href{https://commons.wikimedia.org/wiki/File:Petersen1_tiny.svg}{on Wikicommons} under the same license.  All the other pictures were created by Egan.

\vskip 2em

\section{Maximal Abelian Covers}
\label{sec:maximal_abelian_covers}

We start by reviewing some concepts from Sunada's \textit{Topological Crystallography} \cite{SunadaBook}.   It is convenient to work with graphs having two copies of each edge, one pointing in each direction.   Thus, he defines a \define{graph} \(X = (E,V,s,t,i) \) to consist of a set \(V\) of \define{vertices}, a set \(E\) of \define{edges}, maps \(s,t \maps E \to V\) assigning to each edge its \define{source} and \define{target}, and a map \(i \maps E \to E\) sending each edge to its \define{inverse}, obeying
\[       s(i(e)) = t(e), \quad t(i(e)) = s(e)  ,  \qquad  i(i(e)) = e  \]
and 
\[          i(e) \ne e \]
for all \(e \in E\).  If \(s(e) = v\) and \(t(e) = w\) we write \(e \maps v \to w\), and draw \(e\) as
an interval with an arrow on it pointing from \(v\) to \(w\).  We write \(i(e)\) as \(e^{-1}\), and draw \(e^{-1}\) as the same interval as \(e\), but with its arrow reversed.   The equations obeyed by \(i\) say that taking the inverse of \(e \maps v \to w\) gives an edge \(e^{-1} \maps w \to v\) and that \((e^{-1})^{-1} = e\).   No edge can be its own inverse. 

A \define{map of graphs}, say \(f \maps X \to X'\), is a pair of functions, one sending vertices to vertices and one sending edges to edges, that preserve the source, target and inverse maps.  By abuse of notation we call both of these functions \(f\).

From a graph \(X\) we can build a topological space \(|X|\) called its \define{geometric realization}.  We do this by taking one point for each vertex and gluing on one copy of \([0,1]\) for each edge \(e \maps v \to w\), gluing the point \(0\) to \(v\) and the point \(1\) to \(w\), and then identifying the interval for each edge \(e\) with the interval for its inverse by means of the map \(t \mapsto 1 - t\).   Any map of graphs gives rise to a continuous map between their geometric realizations, and we say a map of graphs is a \define{cover} if this continuous map is a covering map.   For simplicity we denote the fundamental group of \(|X|\) by \(\pi_1(X)\), and similarly for other topological invariants of \(|X|\).   However, there are contexts in which we need to distinguish between a graph \(X\) and its geometric realization \(|X|\).   

Any connected graph \(X\) has a universal cover, meaning a connected cover
\[      p \maps \widetilde{X} \to X \]
that covers every connected cover of \(X\).  The geometric realization of \(\widetilde{X}\) is connected and simply connected.  The fundamental group \(\pi_1(X)\) acts as deck transformations of \(\widetilde{X}\), that is, invertible maps \(g \maps \widetilde{X} \to \widetilde{X}\) such that \(p \circ g = p\).   We can take the quotient of \(\widetilde{X}\) by the action of any subgroup \(G \subseteq \pi_1(X)\) and get a cover \( q \maps \widetilde{X}/G  \to X\).

In particular, if we take \(G\) to be the commutator subgroup of \(\pi_1(X)\), we call the graph \(\widetilde{X}/G\) the \define{maximal abelian cover} of the graph \(X\), and 
denote it by \(\overline{X}\).  We obtain a cover
\[ q \maps \overline{X} \to X  \]
whose group of deck transformations is the abelianization of \(\pi_1(X)\).
This is just the first homology group \(H_1(X,\Z)\), which is a free abelian 
group.

We need a concrete description of the maximal abelian cover.  We start with the universal cover, but first we need some preliminaries on paths in graphs.

Given vertices \(x,y\) in \(X\), define a \define{path} from \(x\) to \(y\) to be a word of edges \(\gamma = e_1 \cdots e_\ell\) with \(e_i \maps v_{i-1} \to v_i\) for some vertices \(v_0, \dots, v_\ell\) with \(v_0 = x\) and \(v_\ell = y\).   We allow the word to be empty if and only if \(x = y\); this gives the \define{trivial path} from \(x\) to itself.  Given a path \(\gamma\) from \(x\) to \(y\) we write \(\gamma \maps x \to y\), and we write the trivial path from \(x\) to itself as \(1_x \maps x \to x\).   We define the composite of paths \( \gamma \maps x \to y \) and \( \delta \maps y \to z \) via concatenation of words, obtaining a path we call \(\gamma \delta \maps x \to z\).    We call a path from a vertex \(x\) to itself a \define{loop} based at \(x\).

We say two paths from \(x\) to \(y\) are \define{homotopic} if one can be obtained from the other by repeatedly introducing or deleting subwords of the form \(e_i e_{i+1}\) where \(e_{i+1} = e_i^{-1}\).   If \([\gamma]\) is a homotopy class of paths from \(x\) to \(y\), we write \( [\gamma] \maps x \to y\).  We can compose homotopy classes \( [\gamma] \maps x \to y\) and \([\delta] \maps y \to z\) by setting \( [\gamma] [\delta] = [\gamma \delta]\).   

If \(X\) is a connected graph, we can describe the universal cover \(\widetilde{X}\) as follows.  Fix a vertex \(x_0\) of \(X\), which we call the \define{basepoint}.  The vertices of \(\widetilde{X}\) are defined to be the homotopy classes of paths \( [\gamma] \maps x_0 \to x\) where \(x\) is arbitrary.   The edges in \(\widetilde{X}\) from the vertex \([\gamma] \maps x_0 \to x\) to the vertex \([\delta] \maps x_0 \to y\) are defined to be the edges \(e \in E\) with \([\gamma e] = [\delta]\).  In fact, there is always at most one such edge.   There is an obvious map of graphs 
\[            p \maps \widetilde{X} \to X \]
sending each vertex  \( [\gamma] \maps x_0 \to x\) of \(\widetilde{X}\) to the vertex
\(x\) of \(X\).  This map is a cover.
 
Now we are ready to construct the maximal abelian cover \(\overline{X}\).  For this, we impose a further equivalence relation on paths, which is designed to make composition commutative whenever possible.  However, we need to be careful.   If \( \gamma \maps x \to y \) and \( \delta \maps x' \to y' \), the composites  \( \gamma \delta \) and \(\delta \gamma\) are both well-defined if and only if \(x' = y\) and \(y' = x\).  In this case, \(\gamma \delta\) and \(\delta \gamma\) share the same starting point and share the same ending point if and only if \(x = x'\) and \(y = y'\).   If all four of these equations hold, both \(\gamma\) and \(\delta\) are loops based at \(x\).    So, we shall impose the relation \(\gamma \delta = \delta \gamma\) only in this case.

We say two paths are \define{homologous} if one can be obtained from another by:

\begin{itemize} 
\item repeatedly introducing or deleting subwords \(e_i e_{i+1}\) where 
\(e_{i+1} = e_i^{-1}\), and/or
\item repeatedly replacing subwords of the form \(e_i \cdots e_j e_{j+1} \cdots e_k \)
by those of the form \break \( e_{j+1} \cdots e_k e_i \cdots e_j \), where
\( e_i \cdots e_j \) and \(e_{j+1} \cdots e_k \) are loops based at the same vertex.
\end{itemize}

\noindent Our use of the term `homologous' is a bit nonstandard, but we shall see some nice relations to homology theory, which we hope justifies this terminology.

We denote the homology class of a path \(\gamma\) by \(\llbracket \gamma \rrbracket\).   Note that if two paths \(\gamma \maps x \to y\), \(\delta \maps x' \to y'\) are homologous then \(x = x'\) and \(y = y'\).   Thus, the starting and ending points of a homology class of paths are well-defined, and given any path \( \gamma \maps x \to y \) we write \( \llbracket \gamma \rrbracket \maps x \to y \).   The composite of homology classes is also well-defined if we set \( \llbracket \gamma \rrbracket \llbracket \delta \rrbracket = \llbracket \gamma \delta \rrbracket\).  

We construct the maximal abelian cover of a connected graph \(X\) just as we constructed its universal cover, but using homology classes rather than homotopy classes of paths.  Fix a basepoint \(x_0\) for \(X\).  The vertices of \(\overline{X}\), or \define{atoms}, are defined to be the homology classes of paths \( \llbracket\gamma\rrbracket \maps x_0 \to x\) where \(x\) is arbitrary.   Any edge of \(\overline{X}\), or \define{bond}, goes from some atom \(\llbracket \gamma\rrbracket \maps x_0 \to x\) to some atom \(\llbracket \delta \rrbracket \maps x_0 \to y\).  The bonds from \(\llbracket \gamma\rrbracket\) to \(\llbracket \delta \rrbracket \) are defined to be the edges \(e \in E\) with \(\llbracket \gamma e \rrbracket = \llbracket \delta \rrbracket\).  There is at most one bond between any two atoms.  Again we have a covering map
\[            q \maps \overline{X} \to X .\]

The homotopy classes of loops based at \(x_0\) form a group, with composition as the group operation.  This is the  \define{fundamental group} \(\pi_1(X)\) of the graph \(X\).  This is isomorphic to the fundamental group of the space associated to \(X\).   By our construction of the universal cover, \(\pi_1(X)\) is also the set of vertices of \(\widetilde{X}\) that are mapped to \(x_0\) by \(p\).  Furthermore, any element \( [\gamma] \in \pi_1(X) \) defines a deck transformation of \(\widetilde{X}\) that sends each vertex \( [\delta] \maps x_0 \to x \) to the vertex \( [\gamma] [\delta] \maps x_0 \to x\).

Similarly, the homology classes of loops based at \(x_0\) form a group with composition as the group operation.   Since the additional relation used to define homology classes is precisely that needed to make composition of homology classes of loops commutative, this group is the abelianization of \(\pi_1(X)\).  It is therefore isomorphic to the first homology group \(H_1(X,\mathbb{Z})\) of the geometric realization of \(X\).  By our construction of the maximal abelian cover, \(H_1(X,\Z)\) is also the set of vertices of \(\overline{X}\) that are mapped to \(x_0\) by \(q\).  Furthermore, any element \( \llbracket\gamma\rrbracket \in H_1(X,\Z) \) defines a deck transformation of \(\overline{X}\) that sends each vertex \( \llbracket\delta\rrbracket \maps x_0 \to x \) to the vertex \( \llbracket\gamma\rrbracket \llbracket\delta\rrbracket \maps x_0 \to x\).

\section{Atoms}
\label{sec:atoms}

Given a connected graph \(X\), we call the vertices of its maximal abelian cover `atoms' because they play the role of atoms in a topological crystal.  Now we describe a systematic procedure for mapping these atoms into a vector space with inner product, namely the space of 1-cycles \(Z_1(X,\R)\).  We show that this map is an embedding if and only if the graph has no bridges.

We begin with some standard material.  Let \(X\) be a graph.  The group of \define{integral 0-chains} on \(X\), \(C_0(X,\Z)\), is the free abelian group on the set of vertices of \(X\).   The group of  \define{integral 1-chains} on \(X\), \(C_1(X,\Z)\), is the quotient of the free abelian group on the set of edges of \(X\) by relations \(e^{-1} = -e\) for every edge \(e\).   The \define{boundary map} is the homomorphism 
\[   \partial \maps C_1(X,\Z) \to C_0(X,\Z)  \]
such that
\[      \partial e = t(e) - s(e) \]
for each edge \(e\), and
\[    Z_1(X,\Z) =  \ker \partial \]
is the group of \define{integral 1-cycles} on \(X\).  Any path \(\gamma = e_1 \cdots e_n\) in \(X\) determines an integral 1-chain \(c_\gamma = e_1 + \cdots + e_n\).     
For any path \(\gamma\) we have \(c_{\gamma^{-1}} = -c_{\gamma}\), and if \(\gamma\) and \(\delta\) are composable then \(c_{\gamma \delta} = c_\gamma + c_\delta\).

We define vector spaces of \define{0-chains} and \define{1-chains} by 
\[C_0(X,\R) = C_0(X,\Z) \otimes \R, \qquad C_1(X,\R) = C_1(X,\Z) \otimes \R, \]
respectively, and extend the boundary map to a linear map 
\[    \partial \maps  C_1(X,\R) \to C_0(X,\R) . \]
We let \(Z_1(X,\R)\) be the kernel of this linear map, or equivalently,
\[    Z_1(X,\R) = Z_1(X,\Z) \otimes \R  , \]
and we call elements of this vector space \define{1-cycles}.  Since the space of 2-chains is trivial, the space of 1-cycles is isomorphic to the first cohomology with real coefficients, \(H_1(X,\R)\).   Since \(Z_1(X,\Z)\) is a free abelian group, it forms a lattice in the space of 1-cycles.  Any edge of \(X\) can be seen as a 1-chain, and there is a unique inner product on \(C_1(X,\R)\) such that edges form an orthonormal basis (with each edge \(e^{-1}\) counting as the negative of \(e\).)   There is thus an orthogonal projection 
\[      \pi \maps C_1(X,\R) \to Z_1(X,\R) . \]

We now come to the main construction, first introduced by Kotani and Sunada \cite{KotaniSunada}:

\begin{defn}
\label{atom_definition}
Let \(X\) be a connected graph with \(V\) as its set of vertices and a chosen basepoint \(x_0 \in V\).  Let its set of \define{atoms} be
\[   A = \{ \llbracket\alpha\rrbracket : \; \alpha \maps x_0 \to x \; \textrm{ for some } x \in V\}.\]
Define the map 
\[    i \maps A \to Z_1(X,\R)  \]
by
\[    i(\llbracket \alpha \rrbracket) = \pi(c_\alpha) .\]
\end{defn}

\noindent
That \(i\) is well-defined follows from Lemma \ref{chain_lemma} below: homologous paths give the same 1-chain.  The map \(i\) is one-to-one precisely for graphs that have no `bridges':

\begin{defn}
\label{bridge_definition}
An edge of a graph is a \define{bridge} if removing that edge disconnects the graph.  A graph is \define{bridgeless} if none of its edges are bridges.
\end{defn}

 Eon \cite{Eon1} and Kotani--Sunada \cite{KotaniSunada} gave arguments for the following result.
  
\begin{thm} 
\label{atom_theorem}
Let \(X\) be a connected graph.   Then the following are equivalent:
\begin{enumerate}
\item \(X\) is bridgeless.
\item The map \( i \maps A \to Z_1(X,\R)  \) is one-to-one.
\end{enumerate}
\end{thm}

\begin{proof}
The map \(i\) is one-to-one if and only if for any  
atoms \(\llbracket \alpha \rrbracket\) and \(\llbracket \beta \rrbracket\), 
\(i(\llbracket \alpha \rrbracket)  = i(\llbracket \beta \rrbracket)\) implies 
\( \llbracket \alpha \rrbracket = \llbracket \beta \rrbracket\).
Note that \(\gamma = \beta^{-1} \alpha \) is a path in \(X\) with \(c_\gamma = c_{\alpha} - c_\beta\), so
\[    \pi(c_\gamma) = \pi(c_{\alpha} - c_\beta) = 
i(\llbracket \alpha \rrbracket) - i(\llbracket \beta \rrbracket) .\]
Since \(\pi(c_\gamma)\) vanishes if and only if \(c_\gamma\) is orthogonal to every
1-cycle, we have 
\[  c_{\gamma} \textrm{ is orthogonal to every 1-cycle}   \; \iff \;   i(\llbracket \alpha \rrbracket)  = i(\llbracket \beta \rrbracket) .   \]
On the other hand, Lemma \ref{chain_lemma} implies that
\[   c_\gamma = 0 \; \iff \; \llbracket \alpha \rrbracket = \llbracket \beta \rrbracket .\]
Thus, to prove (1)\(\iff\)(2), it suffices to that show that \(X\) has no bridges if and only if every 1-chain \(c_\gamma \) orthogonal to every 1-cycle has \(c_\gamma =0\).  We show this in Lemma \ref{bridge_lemma}.  \end{proof}

\begin{lem}
\label{chain_lemma}  
Let \(X\) be a connected graph.  Two paths \(\alpha, \beta \maps x \to y \) in \(X\) are homologous if and only if \(c_\alpha = c_\beta\).
\end{lem}

\begin{proof} The `only if' direction is clear, since \(\beta\) is homologous to \(\alpha\) precisely when it can be obtained from \(\beta\) by:
\begin{itemize} 
\item repeatedly introducing or deleting a subword \(e_i e_{i+1}\) where 
\(e_{i+1} = e_i^{-1}\), and/or
\item repeatedly replacing a subword of the form \(e_i \cdots e_j e_{j+1} \cdots e_k \)
by a subword of the form \( e_{j+1} \cdots e_k e_i \cdots e_j \), where
\( e_i \cdots e_j \) and \(e_{j+1} \cdots e_k \) are loops based at the same vertex,
\end{itemize}
and these moves do not change the corresponding 1-chain.   

Conversely, suppose \(c_\alpha = c_\beta\).  Then \(\gamma = \beta^{-1}\alpha\) is a loop based at \(x\) with \(c_\gamma = c_\alpha - c_\beta = 0\).   In Section \ref{sec:maximal_abelian_covers} we saw that the group of homology classes of loops based at any vertex is the abelianization of \(\pi_1(X)\), namely \(H_1(X,\Z)\).  For a graph we have \(H_1(X,\Z) \cong Z_1(X,\Z)\), and the cycle associated to any loop \(\gamma\) is 
\(c_\gamma\).   Thus, \(c_\gamma = 0 \) implies that \(\gamma\) is homologous to the trivial loop.   This implies that \(\beta\gamma\) is homologous to \(\beta\).  It is also easy to see that \(\beta\gamma = \beta\beta^{-1}\alpha\) is homologous to \(\alpha\).   Thus \(\alpha\) is homologous to \(\beta\).
\end{proof}

The following lemmas are the technical heart of Theorems \ref{atom_theorem} and \ref{embedding_theorem}.  We need to show that any nonzero 1-chain coming from a path in a bridgeless graph has nonzero inner product with some 1-cycle.  The following lemmas, inspired by an idea of Ilya Bogdanov \cite{Bogdanov}, yield an algorithm for actually constructing such a 1-cycle.  This 1-cycle also has other desirable properties, which will come in handy later.

To state these, let a \define{simple path} be one in which each vertex appears at most once.  Let a \define{simple loop} be a loop \(\gamma \maps x \to x\) in which each vertex except \(x\) appears at most once, while \(x\) appears exactly twice, as the starting point and ending point.  Let the \define{support} of a 1-chain \(c\), denoted \(\supp(c)\), be the set of edges \(e\) such that \(\langle c, e\rangle > 0\).   This excludes edges with \(\langle c, e \rangle = 0 \), but also those with \(\langle c , e \rangle < 0\), which are inverses of edges in the support.  Note that
\[    c = \sum_{e \in \supp(c)} \langle c, e \rangle e  .\]
Thus, \(\supp(c)\) is the smallest set of edges such that \(c\) can be written as a positive linear combination of edges in this set.  Given 1-chains \(c\) and \(c'\), we have \(\supp(c') \subseteq \supp(c)\) if and only if \(\langle c, e\rangle = 0\) implies \(\langle c', e \rangle = 0\) and \(\langle c, e \rangle > 0\) implies \(\langle c', e \rangle \ge 0\).

\begin{lem} 
\label{decomposition_lemma_1}
Let \(X\) be any graph and let \(c\) be an integral 1-cycle on \(X\).  
Then for some \(n\) we can write 
\[     c = c_{\sigma_1} + \cdots +  c_{\sigma_n}  \]
where \(\sigma_i\) are simple loops with \(\supp(c_{\sigma_i}) \subseteq \supp(c)\).
\end{lem}

\begin{proof} 
If \(c = 0\) we have \(c = c_\sigma\) for any trivial loop \(\sigma\).  Otherwise \(\supp(c)\) is nonempty, and we can construct a simple loop \(e_1 e_2 \cdots \) in \(X\) as follows.   

We start by picking any edge \(e_1 \maps v_0 \to v_1\) in \(\supp(c)\) and moving to the vertex \(v_1\).   In general, when we arrive at the vertex \(v_i\), if it is not one we have visited before we pick an edge \(e_{i+1} \maps v_i \to v_{i+1}\) in \(\supp(c)\).    To see that such an edge must exist, note that since \(c\) is a cycle, the sum of the coefficients \(\langle c, e \rangle\) over edges \(e \in \supp(c)\) with \(v_i\) as source must equal the sum of these coefficients over edges in \(\supp(c)\) with \(v_i\) as target.   Since we are assuming \(v_i\) is the target of an edge in \(\supp(c)\), namely \(e_i\), it must also be the source of an edge in \(\supp(c)\). 

Since the support of \(c\) is finite, eventually we must reach a vertex we have visited before.  So, eventually \(v_i = v_j\) for some \(j < i\).  Then \(\sigma_1 = e_{j+1} \cdots e_i\) is a simple loop.   By construction we have \(\supp(c_{\sigma_1}) \subseteq \supp(c)\) as desired.   

Now let \(c_2 = c - c_{\sigma_1}\).  We claim that also  \(\supp(c_2) \subseteq \supp(c)\).  To see this, note that both \(\langle c, e \rangle\) and \(\langle c_{\sigma_1}, e\rangle\) vanish except for edges in \(\supp(c)\) and their inverses.  If \(e \in \supp(c)\) then \(\langle c, e\rangle \ge 1\), while \(\langle c_{\sigma_1}, e\rangle \le 1\).  Thus 
\( \langle c, e \rangle = 0\) implies \(\langle c_2, e \rangle = 0\), while \(\langle c, e \rangle \ge 0\) implies \(\langle c_2, e \rangle \ge 0\).    It follows that \(\supp(c_2) \subseteq \supp(c)\).

Since \(c_2\) is again an integral 1-cycle, either \(c_2 = 0\) or we can repeat the argument just given with \(c_2\) replacing \(c\).  Continuing on, we can construct 1-cycles \(c_2, c_3 \dots,\) each with support contained in the support of the previous one, and simple loops \(\sigma_1, \sigma_2, \dots\) with \(c_{k+1} = c_k - c_{\sigma_k}\) and \(\supp(c_{\sigma_k}) \subseteq \supp(c_k)\).  With each new 1-cycle the coefficient of at least one edge is less than before, so eventually \(c_k\) reaches zero and the procedure terminates.  We thus obtain simple loops \(\sigma_1, \dots, \sigma_n\) with \(c = c_{\sigma_1} + \cdots+  c_{\sigma_n}\) and \(\supp(c_{\sigma_i}) \subseteq \supp(c)\).
\end{proof}

\begin{lem} 
\label{decomposition_lemma_2}
Let \(\gamma\maps x \to y\) be a path in a graph \(X\).  Then for some \(n \ge 0\) we can write 
\[     c_\gamma = c_\delta + c_{\sigma_1} + \cdots +  c_{\sigma_n}  \]
where \(\delta\maps x \to y\) is a simple path and \(\sigma_i\) are simple loops with \(\supp(c_\delta), \supp(c_{\sigma_i}) \subseteq \supp(c_\gamma)\).
\end{lem}

\begin{proof} 
Suppose that \(\gamma \maps x \to y\) is a path in \(X\).   If \(x = y\) then \(c_\gamma\) is a cycle and the lemma follows from Lemma \ref{decomposition_lemma_1}, taking \(\delta\) to be any trivial path.  Thus we assume \(x \ne y\).

We claim there is a simple path \(\delta \maps x \to y\) with \(\supp(c_\delta) \subseteq \supp(c_\gamma)\).  We can construct such a path, say \(\delta = e_1 e_2 \cdots \), as follows.  Start by setting \(k = 1\) and let \(c_k = c_\gamma\) in this case.   As we proceed we will repeatedly increment \(k\) and subtract a 1-cycle from \(c_k\) to define 
a new 1-chain \(c_{k+1}\) with \(\supp(c_{k+1}) \subseteq \supp(c_k)\).

We start at the vertex \(v_0 = x\).  In general, if \(v_i\) is neither \(y\) nor a vertex we have visited before, we pick an edge \(e_{i+1} \maps v_i \to v_{i+1}\) in \(\supp(c_k)\) and move to the vertex \(v_{i+1}\).  Such an edge must exist, because \(\partial(c_k) = y - x\), so the sum of the coefficients \(\langle c_k, e\rangle\) over edges \(e \in \supp(c_k)\) with \(v_i\) as source is greater than or equal to the sum of these coefficients over edges in \(\supp(c_k)\) with \(v_i\) as target.   

Since the support of \(c_k\) is finite, eventually \(v_i\) is either \(y\) or a vertex we have visited before.   If \(v_i = y\) we stop: by construction \(\delta = e_1 \cdots e_i\) is a simple path from \(x\) to \(y\).  Moreover, \(\supp(c_\delta) \subseteq \supp(c_k) \subseteq \supp(c_\gamma)\) as desired.

If \(v_i\) is a vertex we have visited before, say \(v_i = v_j\) for \(j < i\), then \(\sigma_k = e_{j+1} \cdots e_i\) is a simple loop.   By construction we have 
\(\supp(c_{\sigma_k}) \subseteq \supp(c_k)\).  The argument given in Lemma \ref{decomposition_lemma_1} shows that \(c_{k+1} = c_k - c_{\sigma_k}\) has \(\supp(c_{k+1}) \subseteq \supp(c_k)\).   Since \(\partial(c_{k+1}) = \partial(c_k) = y - x\), the new 1-chain \(c_{k+1}\) is still nonzero.  We may thus increment \(k\) by 1 and restart the process of building a simple path from \(x\) to \(y\).

This algorithm must eventually succeed in building a simple path \(\delta \maps x \to y\), since the sum of the coefficients of \(c_{k+1}\) is strictly less than that for \(c_k\), yet \(c_{k+1}\) can never vanish, given that \(\partial c_{k+1}= y - x \ne 0\).

If we let \(c = c_\gamma - c_\delta\) then \(c\) is a 1-cycle, and our construction ensures that \(\supp(c) \subseteq \supp(c_\gamma)\).  Thus, we can use Lemma \ref{decomposition_lemma_1} to write
\(c = c_{\sigma_1} + \cdots +  c_{\sigma_n}  \)
where \(\sigma_i\) are simple loops with \(\supp(c_{\sigma_i}) \subseteq \supp(c_\gamma)\).  This implies that
\[     c_\gamma = c_\delta + c_{\sigma_1} + \cdots +  c_{\sigma_n}  \]
where \(\delta\) and \(\sigma_i\) have the desired properties.
\end{proof}

\begin{lem}
\label{bridge_lemma}
Let \(X\) be a graph.   Then the following are equivalent:
\begin{enumerate}
\item \(X\) has no bridges.
\item For any path \(\gamma\) in \(X\), if \(c_\gamma\) is orthogonal to every 1-cycle then \(c_\gamma = 0\).
\end{enumerate}
\end{lem}

\begin{proof} 
To prove (2)\(\implies\)(1), we assume \(X\) has a bridge \(e\) and show \(\langle c, e \rangle = 0\) for every 1-cycle \(c\).  This will give a path \(\gamma = e\) for which 
\(c_\gamma \ne 0\) but \(c_\gamma\) is orthogonal to every 1-cycle.

Since every 1-cycle is a linear combination of 1-cycles \(c_\lambda\) coming from loops \(\lambda\), it suffices to show \(\langle c_\lambda, e \rangle = 0\) for every loop \(\lambda\).  Since \(e\) is a bridge, removing \(e\) breaks \(X\) into two connected components.   Any loop \(\lambda = e_1 \cdots e_n\) must start and end in the same component, so the number of edges \(e_i\) that equal \(e\) must equal the number of \(e_i\) that equal \(e^{-1}\).  Thus \( \langle c_\lambda, e \rangle = 0\).

To prove (1)\(\implies\)(2), suppose \(X\) has no bridges.  It suffices to show that if \(\gamma \maps x \to y\) is a path for which \(c_\gamma\) is orthogonal to every 1-cycle, then \(c_\gamma = 0\).    We use Lemma \ref{decomposition_lemma_2} to write 
\[     c_\gamma = c_\delta + c_{\sigma_1} + \cdots +  c_{\sigma_n}  \]
where \(\delta\maps x \to y\) is a simple path and \(\sigma_i\) are simple loops with \(\supp(c_\delta), \supp(c_{\sigma_i}) \subseteq \supp(c_\gamma)\).   If any of the 1-chains \(c_{\sigma_i}\) are nonzero we are done, since then
\[   \langle c_{\sigma_i}, c_\gamma \rangle \ge  \langle c_{\sigma_i}, c_{\sigma_i} \rangle > 0 \]
so \(c_\gamma\) is not orthogonal to every 1-cycle.  Thus, we assume \(c_\gamma = c_\delta\).  We assume that \(c_\delta \ne 0\), and need to construct a 1-cycle that is not
orthogonal to \(c_\delta\).   We also assume \(x \ne y\), since otherwise we can use 
\(c_\delta\) itself as the desired 1-cycle.

Write \(\delta = e_1 \cdots e_n \) with \(e_i \maps v_{i-1} \to v_i\).  These edges are distinct since \(\delta\) is a simple path, and there is at least one of them since \(x = v_0\) is not equal to \(y = v_n\).  Since the last edge \(e_n\) is not a bridge, there must be a path in \(X\) from \(v_n\) to \(v_0\) that does not include \(e_n\).  If we follow that path only as far as the first vertex \(v_i \ne v_n\) that it reaches, we obtain a path \(\alpha \maps v_n \to v_i\) that includes no edges \(e_i\), nor their inverses.  Thus, we have
\[    \langle c_\alpha, c_\delta \rangle = 0 .\]
If \(\beta = e_{i+1} \cdots e_n \) is the portion of \(\delta\) that goes from \(v_i\) to 
\(v_n\), then 
\[ \langle c_\beta, c_\delta \rangle = \langle c_\beta, c_\beta \rangle \ge 1. \]
The path \(\alpha\beta\) is a loop, so \(c_\alpha + c_\beta\) is a 1-cycle, and this
1-cycle is not orthogonal to \(c_\delta\), since
\[    \langle c_\alpha, c_\delta \rangle + 
\langle c_\beta , c_\delta\rangle \ge 1 . \qedhere \]
\end{proof}

\section{Topological crystals}
\label{sec:topological_crystals}

In the previous section we took a connected bridgeless graph \(X\) and embedded its atoms into the space of 1-cycles via a map
\[      i \maps A \to Z_1(X,\R)  .\]
These atoms are the vertices of the maximal abelian cover \(\overline{X}\), or equivalently, the vertices of its geometric realization \(|\overline{X}|\).    Now we 
extend \(i\) to an embedding 
\[       j \maps |\overline{X}| \to Z_1(X,\R) . \]
We call the image of this embedding the \define{topological crystal} associated to \(X\).

The idea is that just as \(i\) maps each atom to a point in the vector space \(Z_1(X,\R)\), \(j\) maps each edge of \(|\overline{X}|\) to a straight line segment between such points.  These line segments serve as the `bonds' of a topological crystal.  The only challenge is to show that these bonds do not cross each other.  

Let \(X\) be a connected bridgeless graph, and fix a vertex \(x_0\).  Technically, an 
atom is a homology class of paths \(\llbracket \alpha \rrbracket\) starting at \(x_0\) and 
ending at any vertex \(x\).   Given an edge \(e \maps x \to y\) in \(X\), there is a unique edge in \(\overline{X}\) from the atom \(\llbracket \alpha \rrbracket \maps x_0 \to x\) to the atom \(\llbracket \alpha e \rrbracket \maps x_0 \to y\).   Every edge of \(\overline{X}\) arises in this manner in a unique way.  So, an edge in \(\overline{X}\) amounts to a pair
\[   \llbracket \alpha \rrbracket \maps x_0 \to x, \quad e \maps x \to y .\]

As explained in Section \ref{sec:maximal_abelian_covers}, each edge in \(\overline{X}\) gives a closed interval in the geometric realization \(|\overline{X}|\), with a preferred parametrization \(\phi \maps [0,1] \to |\overline{X}|\).   The inverse of an edge gives the same closed interval with the reverse parametrization, \(1 - \phi\).   We call these closed
intervals \define{bonds} to distinguish from the edges of an abstract graph. 

Given a finite-dimensional vector space \(V\), we say a map \( f \maps |\overline{X}| \to V\) is \define{affine on bonds} if for any given bond the composite map \(f \circ \phi \maps [0,1] \to V\) is affine, meaning
\[       f(\phi(t)) = a + b t \] 
for some \(a,b \in V\).  It follows that \(f\) maps each bond to a straight line segment in \(V\).  It also follows that \(f\) is continuous.   Note also that given \(f\) defined on the atoms, there exists a unique map extending \(f\) that is affine on bonds.   The reason is that an affine map from \([0,1]\) to \(V\) is determined by its values at the endpoints.

\begin{thm} 
\label{embedding_theorem}
If \(X\) is a connected graph, the map \(i \maps A \to Z_1(X,\R)\) extends uniquely to a map 
\[   j \maps |\overline{X}| \to Z_1(X,\R) \] 
that is affine on bonds.  If \(X\) is also bridgeless, then \(j\) is one-to-one.
\end{thm}

\begin{proof}
As mentioned, the existence and uniqueness of an extension \(j\) that is affine on bonds is automatic.  The task is to show that \(j\) is an embedding when \(X\) has no bridges.

Each bond in \(\overline{X}\) arises from a pair 
\[   \llbracket \alpha \rrbracket \maps x_0 \to x, \quad e \maps x \to y ,\]
and the points on this bond are parametrized by numbers \(t \in [0,1]\).  Thus, 
we can describe any point \(p \in |\overline{X}|\) as a triple 
\[ \llbracket\alpha\rrbracket \maps x_0 \to x,  \quad e \maps x \to y, \quad t \in [0,1].\]     
However, there is some redundancy in this description.  If \(t = 0\) then \(p\) is
an atom, and in this case \(e\) becomes irrelevant: any choice of
\(e \maps x \to y\) gives the same point \(p\).  In general, any point described by the
triple 
\[ \llbracket\alpha\rrbracket \maps x_0 \to x,  \quad e \maps x \to y, \quad t \in [0,1] \]
is also described by the triple
\[ \llbracket\alpha e \rrbracket \maps x_0 \to y,  \quad e^{-1} \maps y \to x, \quad 1-t \in [0,1].\]
If \(t \in (0,1)\), the point \(p\) is not an atom, and the above triples are the only two
that describe it.

Suppose we have a point \(p \in |\overline{X}|\) corresponding to some triple
\(\llbracket\alpha\rrbracket \maps x_0 \to x\), \(e \maps x \to y\), \(t \in [0,1]\). We define a 1-chain \(c_p \in C_1(X,\R)\) by
\[   c_p = c_\alpha + t e .\]
Thanks to Lemma \ref{chain_lemma}, \(c_p\) is independent of the choice of path \(\alpha\) representing the atom \(\llbracket\alpha\rrbracket\).   Furthermore, different triples describing the same point \(p\) give the same result for \(c_p\).  Thus, \(c_p\) 
is well-defined.  

Define
\[   j \maps |\overline{X}| \to Z_1(X,\R)  \]
by
\[    j(p) = \pi(c_p) .\]
This equals \(i(p)\) when \(p\) is an atom, since \(c_p = c_\alpha\) when \(p\) is the atom \(\llbracket \alpha \rrbracket \).  The map \(j\) is affine on each bond since \(c_p\) depends affinely on the parameter \(t\) and \(\pi\) is linear.  So, \(j\) is the unique
map equalling \(i\) on atoms that is affine on each bond.  

To prove the theorem, it suffices to show that if \(p\) and \(q\) are points in 
\(\overline{X}\) such that \(c_p - c_q\) is orthogonal to every 1-cycle, then \(p = q\).    
To do this, first we show that \(c_p = c_q\) implies \(p = q\).  Then we show
that if \(c_p - c_q\) is orthogonal to every 1-cycle then \(c_p = c_q\).

Suppose \(c_p = c_q\).  These 1-chains are integral if and only if \(p\) and \(q\) are atoms, in which case \(p = q\) by Lemma \ref{chain_lemma}.   So, we may assume that neither 1-chain is integral and neither \(p\) nor \(q\) is an atom.  Thus, we may 
write \(p\) as a triple 
\[ \llbracket\alpha\rrbracket,  \quad e , \quad t \in (0,1) \]
and \(q\) as a triple
\[ \llbracket\beta\rrbracket,  \quad f , \quad u \in (0,1) , \]
implying
\[   c_\alpha + te = c_\beta + uf .\]
Working modulo integral 1-chains we have 
\[  te = uf  \bmod C_1(X,\Z),  \]
so either \(f = e\) and \(u = t\) or \(f = e^{-1}\) and \(u = 1-t\).
In the first case we have
\[   c_\alpha = c_\beta  \]
so \(\llbracket \alpha \rrbracket = \llbracket \beta \rrbracket \) by Lemma \ref{chain_lemma}, and thus \(p = q\).  In the second case we have 
\[   c_\alpha + e = c_\beta \]
so \(\llbracket \alpha e \rrbracket = \llbracket \beta \rrbracket \) by Lemma \ref{chain_lemma}, and again \(p = q\).

Next suppose that \(c_p - c_q\) is orthogonal to every 1-cycle.  Write
\[     c = c_p - c_q \]
and note that
\[     c = c_\alpha - c_\beta + t e - u f .\]
We need to prove that if \(c\) is orthogonal to every 1-cycle then \(p = q\).

Form a path \(\gamma\) by either composing \(\beta^{-1}\alpha\) with the 
edges \(f^{-1}\) and \(e\), or not, in such a way that the coefficient of every edge in 
\(c_{\gamma}\) agrees in sign with its coefficient in \(c\).   More precisely, let
\[  \gamma = \left\{\begin{array}{cc} 
f^{-1} \beta^{-1}\alpha e    & \textrm{if } 
\langle c, -f \rangle > 0 \textrm{ and } \langle c, e \rangle > 0 \\
f^{-1} \beta^{-1}\alpha       & \textrm{if } 
 \langle c, -f \rangle > 0 \textrm{ and } \langle c, e \rangle \le 0 
\\ 
\beta^{-1}\alpha e      &  \textrm{if } 
\langle c, -f \rangle \le 0 \textrm{ and } \langle c, e \rangle > 0 
\\ 
\beta^{-1}\alpha         & \textrm{if } 
\langle c, -f \rangle \le 0 \textrm{ and } \langle c, e \rangle \le 0 .
\end{array} \right.
\] 
Thus \(\supp(c_\gamma) = \supp(c)\).     Next, using Lemma \ref{decomposition_lemma_2}, write
\[     c_\gamma = c_\delta + c_{\sigma_1} + \cdots +  c_{\sigma_n}  \]
where \(\delta\maps x \to y\) is a simple path and \(\sigma_i\) are simple loops with \(\supp(c_\delta), \supp(c_{\sigma_i}) \subseteq \supp(c_\gamma)\).

If any 1-cycle \(c_{\sigma_i}\) is nonzero, then since \(\supp(c_{\sigma_i}) \subseteq
\supp(c_\gamma) = \supp(c)\) we have
\[    \langle c, c_{\sigma_i} \rangle  > 0 ,\]
contradicting our assumption that \(c\) is orthogonal to every 1-cycle.   We thus assume there are no nonzero 1-cycles \(c_{\sigma_i}\).  This implies \(c_\gamma = c_\delta\).  
If \(\delta\) is a path with no edges, then \(c_\gamma = 0\).  Since \(\supp(c_\gamma) = \supp(c)\), this implies \(c = 0\) as desired.  

We are left with the case where \(\delta\) has at least one edge.  However, we shall see that this leads to a contradiction.  Write \(\delta = e_1 \cdots e_n \) with \(e_i \maps v_{i-1} \to v_i\).  There is at least one of these edges, and they are distinct since \(\delta\) is a simple path.  Since the last edge \(e_n\) is not a bridge, there must be a path in \(X\) from \(v_n\) to \(v_0\) that does not include \(e_n\).  If we follow that path only as far as the first vertex \(v_i \ne v_n\) that it reaches, we obtain a path \(\alpha \maps v_n \to v_i\) that includes no edges \(e_i\), nor their inverses.  Since \(\supp(c) = \supp(c_\delta)\), this implies
\[    \langle c_\alpha, c \rangle = 0 .\]
If \(\beta = e_{i+1} \cdots e_n \) is the portion of \(\delta\) that goes from \(v_i\) to 
\(v_n\), then \( \langle c_\beta, c_\delta \rangle \ge 1 \), and since \(\supp(c) = \supp(c_\delta)\), we have
\[  \langle c_\beta, c \rangle > 0 .\]
Since  \(\alpha\beta\) is a loop, \(c_\alpha + c_\beta\) is a 1-cycle, and
\[    \langle c_\alpha + c_\beta , c \rangle > 0 .\]
This contradicts our assumption that \(c\) is orthogonal to every 1-cycle.
\end{proof}

\section{Symmetries} 
\label{sec:symmetries}

Every bridgeless graph gives a topological crystal.  How are the graph's symmetries related to those of its crystal?  To tackle this, we start by asking how the symmetries of a graph 
\(X\) are related to those of its maximal abelian cover 
\[      q \maps \overline{X} \to X  .\]
We shall not study all the symmetries of \(\overline{X}\), only those that cover symmetries 
of \(X\).  These form a group \(\Cov(X)\).  As we shall see, this group contains \(H_1(X,
\Z)\), which acts as deck transformations of \(\overline{X}\).  The group \(\Cov(X)\) also 
maps onto the group of symmetries of \(X\): in other words, every symmetry of \(X\) can 
be covered by some symmetry of \(\overline{X}\). 

In fact, we shall prove that \(\Cov(X)\) is an extension of the symmetry group of \(X\) by \(H_1(X,\Z)\).  We also prove that \(\Cov(X)\) acts on the vector space \(Z_1(X,\R)\) as affine isometries in a way that preserves the topological crystal embedded in this space.

Define an \define{automorphism} of a graph \(X\) to be a map of graphs \(f \maps X \to X\) that has an inverse.   The automorphisms of \(X\) form a group \(\Aut(X)\).   We say an automorphism \(g \in \overline{X} \to \overline{X}\) \define{covers} an automorphism \(f \maps X \to X\) if 
\[         q \circ g = f \circ q .\]
There is a group \(\Cov(X)\) where an element is an automorphism of \(\overline{X}\) that covers some automorphism of \(X\).  We call elements of \(\Cov(X)\) \define{covering symmetries}.  Since \(q\) is onto, any covering symmetry covers at most one automorphism of \(X\).  There is thus a map
\[    \psi \maps \Cov(X) \to \Aut(X)  \]
sending any \(g \in \Cov(X)\) to the automorphism \(f \in \Aut(X)\) that it covers.
It is easy to check that \(\psi\) is a group homomorphism.

\begin{lem}
\label{covering_lemma_1}
Let \(X\) be a connected graph with basepoint.  The homomorphism \(\psi\) is onto 
and its kernel is \(H_1(X,\Z)\), so we have a short exact sequence
\[       1 \longrightarrow H_1(X,\Z) \longrightarrow \Cov(X) \stackrel{\psi}{\longrightarrow} \Aut(X) \longrightarrow 1 .\]
\end{lem}

\begin{proof}
An element \(g \in \Cov(X)\) is in the kernel of \(\psi\) if and only if it covers the identity map, which means that it is a deck transformation of the maximal abelian cover of \(X\).  In Section \ref{sec:maximal_abelian_covers} we saw that the group of these deck transformations is \(H_1(X,\Z)\).  

To show that \(\psi\) is onto, we choose any automorphism \(f \maps X \to X\) and find
an automorphism \(g \maps \overline{X} \to \overline{X}\) that covers
\(f\).   We need to describe how \(g\) acts on vertices and edges.  A vertex of 
\(\overline{X}\) is an atom, that is, a homology class of paths in \(X\) starting from the basepoint, say \(\llbracket\alpha\rrbracket \maps x_0 \to x\).  The automorphism \(f\) sends any path \(\alpha \maps x_0 \to x\) to a path we call \(f(\alpha) \maps f(x_0) \to f(x)\).  Unfortunately this path no longer defines an atom unless \(f(x_0) = x_0\).   To deal with this, arbitrarily choose a path  \(\beta \maps x_0 \to f(x_0)\) and let
\[  g(\llbracket\alpha\rrbracket) = \llbracket \beta f(\alpha) \rrbracket \maps x_0 \to 
f(x) .\]
An edge \(\overline{X}\) is a pair
\[   \llbracket \alpha \rrbracket \maps x_0 \to x, \quad e \maps x \to y .\]
We let \(g\) map this edge to the edge
\[   g(\llbracket \alpha \rrbracket) \maps x_0 \to f(x), \quad f(e) \maps f(x) \to f(y) .\]

One can check that \(g\) is a map of graphs.  Since it is one-to-one and onto on both vertices and edges, it is an automorphism.    One can also check that \(g\) covers 
\(f\), since the definition of the projection \(q \maps \overline{X} \to X\) implies
that
\[    q( \llbracket \alpha \rrbracket) = x \]
and 
\[    q(g(\llbracket \alpha \rrbracket)) = f(x)  .  \qedhere\] 
\end{proof}

The proof of the above lemma gives an explicit description of the group 
\(\Cov(X)\) and its action on \(\overline{X}\).  In the proof, two choices of 
\(\beta \maps x_0 \to f(x_0)\) define the same automorphism of \(\overline{X}\) if 
and only if they are homologous. A homology class of paths \(\llbracket\beta\rrbracket \maps x_0 \to f(x_0)\) is just an atom mapping to \(f(x_0)\) via the projection 
\(q \maps \overline{X} \to X\).  Thus, we have:

\begin{lem}
\label{covering_lemma_2}
Elements of \(\Cov(X)\) are in one-to-one correspondence with pairs \((f,\llbracket\beta\rrbracket)\) where \(f \in \Aut(X)\) and \(\llbracket \beta \rrbracket \maps x_0 \to f(x_0)\).  Such a pair acts on \(\overline{X}\) by mapping each atom 
\[   \llbracket \alpha \rrbracket \maps x_0 \to x\]
to the atom
\[   \llbracket \beta f(\alpha) \rrbracket \maps x_0 \to f(x) .\]
and each edge
\[   \llbracket \alpha \rrbracket \maps x_0 \to x, \quad e \maps x \to y \]
to the edge
\[   \llbracket \beta f(\alpha) \rrbracket \maps x_0 \to f(x), \quad f(e) \maps f(x) \to f(y) .\qedhere \] 
\end{lem}
 
There may in principle be other automorphisms of \(\overline{X}\) not lying in the group 
\(\Cov(X)\); we have not found examples.  Automorphisms in \(\Cov(X)\) have the advantage that they also act naturally as affine transformations of the vector space 
\(Z_1(X,\R)\).  Furthermore, this space has an inner product
coming from the inner product on 1-chains, which in turn defines
a metric.  The group \(\Cov(X)\) acts as \define{affine isometries} of \(Z_1(X,\R)\), that 
is, affine transformations that preserve this metric:

\begin{thm}
\label{symmetry_theorem_1}
There exists a unique action \(\rho\) of \(\Cov(X)\) as affine transformations of \(Z_1(X,\R)\) for which the embedding \(i \maps A \to Z_1(X,\R)\) is equivariant, meaning that
\[            \rho(g) i(\llbracket \alpha \rrbracket)= i(g\llbracket \alpha \rrbracket) \]
for all \(g \in \Cov(X)\) and \(\llbracket \alpha \rrbracket \in A.\)    Moreover, the transformations \(\rho(g)\) are affine isometries. 
\end{thm}

\begin{proof}
Since \(Z_1(X,\Z) \subseteq A\), every point in \(Z_1(X,\R)\) is an affine combination of 
points \(i(\llbracket \alpha \rrbracket)\) where \(\llbracket \alpha \rrbracket\) is an atom, 
so an affine transformation of \(Z_1(X,\R)\) is uniquely determined by its action on such
points.  This proves the uniqueness of \(\rho\).  

For existence, take any \(g \in \Cov(X)\) and use Lemma \ref{covering_lemma_2} to write it as a pair \((f,\llbracket\beta\rrbracket)\) where \(f \in \Aut(X)\) and \(\llbracket \beta \rrbracket \maps x_0 \to f(x_0) \).   For any \(c \in Z_1(X,\R)\), define
\[     \rho(g) (c) = i(\llbracket\beta\rrbracket) + f_*(c) \]
where
\[     f_* \maps C_1(X,\R) \to C_1(X,\R)  \]
is the linear transformation with \(f_*(e) = f(e)\) for any edge \(e\) of \(X\).
Since \(f_*\) preserves the inner product on 1-chains and the subspace of 1-cycles
we have
\[    f_* \pi = \pi f_* \]
where \(\pi\) is the projection from 1-chains to 1-cycles.  Recall from Definition \ref{atom_definition} that \(i(\llbracket \alpha \rrbracket) = \pi(c_\alpha)\).  Thus 
we have
\[ \begin{array}{ccl}    
\rho(g) (i(\llbracket \alpha \rrbracket)) 
&=&  i(\llbracket\beta\rrbracket)  + f_* (i(\llbracket \alpha \rrbracket)) \\
&=& \pi(c_\beta) + f_*(\pi(c_\alpha)) \\
&=& \pi(c_\beta) + \pi(f_*(c_\alpha)) \\
&=& \pi(c_\beta + c_{f(\alpha)}) \\
&=& \pi(c_{\beta f(\alpha)})  \\
&=&  i(\llbracket \beta f(\alpha) \rrbracket) \\
&=&  i(g\llbracket \alpha \rrbracket)  
\end{array}
\]
where in the fourth step we used the fact that \(f_*(c_\alpha) = c_{f(\alpha)}\) for any path
\(\alpha\).   

To check that \(\rho\) is an action we need to show \(\rho(gh) = \rho(g)\rho(h)\) 
for any \(g,h \in \Cov(X)\).  This can be done by brute force, but it suffices to check 
it on points of the form \(i(\llbracket \alpha \rrbracket)\), since every point of 
\(Z_1(X,\R)\) is an affine combination of these.   For this, note that:
\[ \begin{array}{ccl}
 \rho(gh) (i(\llbracket \alpha \rrbracket)) &=& i( (gh) (\llbracket \alpha \rrbracket)) \\
&=&  i( g(h \llbracket \alpha \rrbracket)) \\
&=& \rho(g) (i (h \llbracket \alpha \rrbracket)) \\
&=& \rho(g)(\rho(h) (i(\llbracket \alpha \rrbracket))).
\end{array}
\]
Finally, is clear that \(\rho(g)\) is an affine isometry, since 
\[     \rho(g) (c) = i(\llbracket\beta\rrbracket) + f_*(c) \]
implies that \(\rho(g)\) is a linear isometry followed by a translation.
\end{proof}

We conclude by showing that \(\Cov(X)\) acts on the space of 1-cycles in a way
that preserves the topological crystal.  In general, any automorphism \(f\) of any graph functorially determines an automorphism of its geometric realization, which we call 
\(|f|\).   Thus, \(\Cov(X)\) acts as automorphisms of
\(|\overline{X}|\).   In Theorem \ref{embedding_theorem} we saw that
the map \(i \maps A \to Z_1(X,\R)\) extends uniquely to an embedding 
\( j \maps |\overline{X}| \to Z_1(X,\R) \) 
that is affine on bonds.   To show that  \(\Cov(X)\) preserves the topological
crystal, that is, the image of this embedding, we prove that \(j\) is equivariant:

\begin{thm}
\label{symmetry_theorem_2}
The embedding \(j \maps |\overline{X}| \to Z_1(X,\R)\) is equivariant with
respect to the action of \(\Cov(X)\), meaning that
\[            \rho(g) j(p)= j(|g|(p)) \]
for all \(g \in \Cov(X)\) and \(p \in |X|.\)  
\end{thm}

\begin{proof}
As noted in the proof of Theorem \ref{embedding_theorem}, every point \(p \in |X|\)
corresponds to a triple 
\[ \llbracket \alpha \rrbracket \maps x_0 \to x, \quad e \maps x \to y, \quad 
t \in [0,1], \] 
and then we have
\[    j(p) = \pi(c_\alpha + te)  .\]
If we let \(\beta = \alpha e \) we thus have
\[   j(p) = (1-t) \pi(c_\alpha) + t \pi(c_\beta) \]
or in other words
\[   j(p) = (1-t) i(\llbracket \alpha \rrbracket) + t i(\llbracket \beta \rrbracket) .\]

By Lemma \ref{covering_lemma_2} any element \(g \in \Cov(X)\) corresponds to some 
pair \(f \in \Aut(X), \llbracket \beta \rrbracket \maps x_0 \to f(x_0)\).  The point \(|g|(p)\)
then corresponds to the triple 
\[ g(\llbracket \alpha \rrbracket) \maps x_0 \to f(x), \quad g(e) \maps f(x) \to f(y), \quad 
t \in [0,1]. \] 
Since \(\rho(g)\) is affine, it follows that
\[ \begin{array}{ccl}
\rho(g) j(p) &=& 
(1-t) \rho(g) ( i(\llbracket \alpha \rrbracket) ) \;+ \; t \rho(g)  i(\llbracket \beta \rrbracket)
\\  
&=& (1-t) i(g(\llbracket \alpha \rrbracket) ) \;+ \; t i(g(\llbracket \beta \rrbracket)) \\ 
&=&  j(|g|(p)).  \qquad\qquad\qquad \qquad\qquad\qquad \qedhere
\end{array} 
\]
\end{proof}

\section{The packing fraction}
\label{sec:packing}

In Section \ref{sec:atoms} we took a connected bridgeless graph \(X\) with a basepoint \(x_0\) and embedded its set of atoms, \(A\), into its space of 1-cycles, \(Z_1(X,\R)\).  Now we shall use this embedding to reinterpret \(A\) as a subset of \(Z_1(X,\R)\), as follows:
\[  A = \{ \pi(c_\alpha) :  \; \alpha \maps x_0 \to x \textrm{ for some } x \}   .\]
It is interesting to ask how densely packed these atoms are.  Computing the density of the corresponding sphere packing seems hard, but we can also ask what fraction of `potential locations for atoms' are actually filled by atoms.  To make sense of this, note that \(A\) 
is contained in the lattice \(L\) obtained by projecting the integral 1-chains down to the space of 1-cycles:
\[      L = \{   \pi(c) : \; c \in C_1(X,\Z)  \} \]
where \(\pi\) is the projection of the space of 1-chains onto the space of 1-cycles.
We can think of points in \(L\) as potential locations for atoms.  

We cannot define the fraction of potential atoms that are actual atoms to be 
\(|A|/|L|\), because both the numerator and denominator in this fraction are infinite.
To deal with this, note that 
\[    Z_1(X,\Z) \subseteq A \subseteq L  \]
and \(Z_1(X,\Z)\) acts as translations on both \(A\) and \(L\).  For \(L\) this is obvious, because \(Z_1(X,\Z)\) is a sublattice of \(L\).  For \(A\) it follows from
\[   Z_1(X,\Z) = \{ \pi(c_\gamma) : \;  \gamma \maps x_0 \to x_0 \}  \]
and the fact that composing paths has the effect of adding their 1-chains: we can 
compose any loop \(\gamma \maps x_0 \to x_0\) with any path 
\(\alpha \maps x_0 \to x\) and get an atom \(\pi(c_{\gamma \alpha}) = 
\pi(c_\gamma) + \pi(c_\alpha)\) which is the result of translating the atom 
\( \pi(c_\alpha) \) by \( \pi(c_\gamma) \in Z_1(X,\Z)\).

When \(X\) is a finite graph, \(A/Z_1(X,\Z)\) and \(L/Z_1(X,\Z)\) are finite abelian groups.
In this case we define the \define{packing fraction} to be the ratio 
\[  \frac{|A/Z_1(X,\Z)|}{|L/Z_1(X,\Z)|}.  \]
This is a regularized version of the meaningless ratio \(|A|/|L|\).

To compute the packing fraction, we need another description of the lattice
\(L\).   Given any lattice \(\Lambda\) in a finite-dimensional real inner product space \(V\), 
the \define{dual lattice} \(\Lambda^*\) is 
\[    \Lambda^* = \{ v \in V : \; \langle v, w \rangle \in \Z \textrm{ for all } w \in \Lambda \}.\]
We say \(\Lambda\) is \define{integral} if  \(\Lambda \subseteq \Lambda^*\), or in
other words, if the inner product of any two vectors in \(\Lambda\) is an integer. We say \(\Lambda\) is \define{self-dual} if \(\Lambda = \Lambda^*\).   

It is easy to see that \(C_1(X,\Z)\) is a self-dual lattice in \(C_1(X,\R)\), since edges
form an orthonormal basis.  It follows that \(Z_1(X,\Z)\) is an integral lattice in 
\(Z_1(X,\R)\), where the latter space inherits its inner product from \(C_1(X,\R)\).
However, \(Z_1(X,\Z)\) is not in general self-dual, and this gives another description
of the lattice \(L\):

\begin{lem} \label{dual_lemma}
If \(X\) is a finite graph then \(L = Z_1(X,\Z)^*\).
\end{lem}

\begin{proof}  This was shown by Bacher, de la Harpe and Nagnibeda \cite[Lemma 1]{BHN}.  
\end{proof}

Using this we can prove:

\begin{thm}
\label{packing_fraction_theorem}
If \(X\) is a finite graph without bridges, then its packing fraction is 
\[ \frac{|A/Z_1(X,\Z)|}{|L/Z_1(X,\Z)|} = \frac{|V|}{|T|}  \]
where \(V\) is the set of vertices of \(X\) and \(T\) is the set of spanning trees in \(X\).
In fact 
\[      |A/Z_1(X,\Z)| = |V|  \]
and 
\[      |L/Z_1(X,\Z)| = |T| .\]
\end{thm}

\begin{proof}
To show that \(|A/Z_1(X,\Z)| = |V|\), we use several facts.  First, \(A\) is the set of
vertices of the maximal abelian cover of \(X\).  Second, \(Z_1(X,\Z) \cong H_1(X,\Z)\).  
Third, via this isomorphism, the action of \(H_1(X,\Z)\) as deck transformations on
vertices of the maximal abelian cover is equivalent to the action of \(Z_1(X,\Z)\) by translations on \(A\).  Thus, the quotient \(A/Z_1(X,\Z)\) is isomorphic to the set of vertices of \(X\).

To show that \(|L/Z_1(X,\Z)| = |T|\) we use the work of Bacher, de la Harpe and Nagnibeda on integral cuts \cite{BHN}.    Besides the already mentioned inner product on 1-chains, there is an inner product on 0-chains for which the vertices of \(X\) form an orthonormal basis.  This lets us define the adjoint
\[            \partial^* \maps C_0(X,\R) \to C_1(X,\R) . \]
Concretely, for any vertex \(x\) we have
\[ \partial^* x = \sum_{e \maps y \to x} e \]
where we sum over all edges with target \(x\).  It is well known \cite{Biggs} that there is an orthogonal direct sum decomposition
\[            C_1(X,\R) = Z_1(X,\R) \oplus N_1(X,\R)  \]
where 
\[             N_1(X,\R) = \im \, \partial^* \]
is called the space of \define{cuts}.   Inside the space of cuts there is a lattice 
\[             N_1(X,\Z) = N_1(X,\R) \cap C_1(X,\Z)   \]
called the lattice of \define{integral cuts}.     However, the decomposition of 1-chains into 
1-cycles and cuts fails at the integral level: 
\[             C_1(X,\Z) \ne Z_1(X,\Z) \oplus N_1(X,\Z) .\]
This is the key to understanding the quotient \(L/Z_1(X,\Z)\).

Since \(C_1(X,\Z)\) is an integral lattice, so are \(N_1(X,\Z)\) and \(Z_1(X,\Z)\).  
If we had \( C_1(X,\Z) = Z_1(X,\Z) \oplus N_1(X,\Z) \) then \(N_1(X,\Z)\) and \(Z_1(X,\Z)\) would be self-dual, but in fact they are usually not.  With the help of Kirchhoff's spanning tree theorem, Bacher, de la Harpe and Nagnibeda \cite[Prop.\ 2]{BHN} prove that
\[    |{N_1(X,\Z)^*/N_1(X,\Z)}| = |T|  .\]
Then they use a very general fact about lattices \cite[Lemma 2]{BHN} to construct isomorphisms
\[       \frac{Z_1(X,\Z)^*}{Z_1(X,\Z)} \cong
         \frac{C_1(X,\Z)}{Z_1(X,\Z) \oplus N_1(X,\Z)} \cong 
          \frac{N_1(X,\Z)^*}{N_1(X,\Z)}.  \]
This implies
\[    |Z_1(X,\Z)^* / Z_1(X,\Z)| = |T| . \]
But \(Z_1(X,\Z)^* = L\) by Lemma \ref{dual_lemma}, so  
\[ |L/Z_1(X,\Z)| = |T|\] 
as desired, and thus 
\[   \frac{|A/Z_1(X,\Z)|}{|L/Z_1(X,\Z)|} = \frac{|V|}{|T|}  .   \qedhere \]
\end{proof}

The concept of `cut' has other applications to topological crystals.   In the proof above we saw that for any vertex \(x\),
\[ \partial^* x = \sum_{e \maps y \to x} e  \]
is a cut, so it projects to zero in \(Z_1(X,\R)\):
\[   \sum_{e \maps y \to x} \pi(e) = 0 .\]
Multiplying by \(-1\), we obtain 
\[    \sum_{e \maps x \to y} \pi(e) = 0 \]
where we sum over all edges with \(x\) as source.
For any atom \(\pi(c_\alpha)\) arising from a path \(\alpha \maps x_0 \to x\), the
bonds coming out of this atom correspond to edges \(e \maps x \to y\), and they
connect it to atoms \(\pi(c_{\alpha e}) = \pi(c_\alpha) + \pi(e)\).   The above
equation thus says that the bonds coming out of any atom give vectors in \(Z_1(X,\R)\)
summing to zero.  

This also follows from Kotani and Sunada's characterization of topological crystals in terms of energy minimization \cite{KotaniSunada, SunadaBook}.  If we think of the bonds as springs, all obeying Hooke's law with the same spring constant, the total force on each atom must vanish in equilibrium, so the bonds must give vectors that sum to zero.

\section{Examples}
\label{sec:examples}

If \(X\) is a graph such that \(\Aut(X)\) acts transitively on `arcs', meaning pairs consisting of a 
vertex and an edge incident to that vertex, then \(\Cov(X)\) acts transitively on the 
arcs of the corresponding topological crystal.  One example of this phenomena is the Laves 
graph, coming from the tetrahedron.   The symmetry group of the tetrahedron is the 
Coxeter group 
\[     \A_3 = \langle s_1, s_2, s_3 \;| \; (s_1s_2)^3 = (s_2s_3)^3 = s_1^2 = s_2^2 = s_3^2 = 1\rangle  . \]
Thus, by Lemma \ref{covering_lemma_1}, the group of covering symmetries of the Laves graph is an extension of \(\A_3\) by \(\Z^3\). 

More generally, the vertices and edges of any Platonic solid form a graph whose symmetry group acts transitively on arcs.  For example, the symmetry group of the cube and octahedron is the Coxeter group
\[     \B_3 = \langle s_1, s_2, s_3 \;| \; (s_1s_2)^3 = (s_2s_3)^4 = s_1^2 = s_2^2 = s_3^2 = 1\rangle  . \]
Since the cube has 6 faces, $H_1(X,\R)$ is 5-dimensional for the cube graph.   The corresponding topological crystal is thus 5-dimensional, and its group of covering symmetries, an extension of \(\B_3\) by \(\Z^5\), acts transitively on arcs.  Similarly, the octahedron gives a 7-dimensional topological crystal whose group of covering symmetries, an extension of \(B_3\) by \(\Z^7\), acts transitively on arcs.  The cuboctahedron can be seen either as a truncated cube or a truncated octahedron, so it too has \(\B_3\) as its symmetry group, and in fact this group acts transitively on arcs.   Since the cuboctahedron has 14 faces, we obtain a 13-dimensional crystal whose covering symmetries act transitively on arcs.

The symmetry group of the dodecahedron and icosahedron is
\[     \H_3 = \langle s_1, s_2, s_3 \;| \; (s_1s_2)^3 = (s_2s_3)^5= s_1^2 = s_2^2 = s_3^2 = 1\rangle  , \]
and these solids give crystals of dimensions 11 and 19.  The icosidodecahedron can be seen either as a truncated dodecahedron or a truncated icosahedron; it has \(\H_3\) as its symmetry group, and this group acts transitively on arcs.    It has 32 faces, so it gives a 31-dimensional crystal whose covering symmetries act transitively on arcs.

There is also an infinite family of degenerate Platonic solids called `hosohedra' with two vertices, \(n\) edges and \(n\) faces.   These faces cannot be made flat, since each face has just 2 edges, but that is not relevant to our construction: the vertices and edges still give a graph.  For example, when \(n = 6\), we have the `hexagonal hosohedron':

\begin{center} 
\raisebox{-0.0 em}{\includegraphics[scale = 0.5]{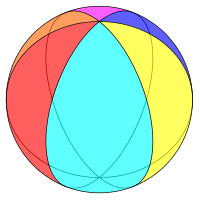}} 
\end{center}

\noindent
The symmetry group of the \(n\)-gonal hosohedron is \(\I_n \times \Z/2\), where the Coxeter group 
\[     \I_n = \langle s_1, s_2 \;| \; (s_1s_2)^n = s_1^2 = s_2^2 = 1\rangle   \]
is a dihedral group, and \(\Z/2\) acts to interchange the two vertices while fixing the
edges.   The automorphism group of its underlying graph is \(S_n \times \Z/2\).  The corresponding crystal has dimension \(n-1\), and its group of covering symmetries is an
extension of \(S_n \times \Z/2\) by \(\Z^{n-1}\).   The case \(n = 3\) gives the graphene crystal, while \(n = 4\) gives the diamond.

There are many other graphs whose symmetry group acts transitively on arcs.  One is the Petersen graph:

\begin{center} 
\raisebox{-0.0 em}{\includegraphics[scale = 0.2]{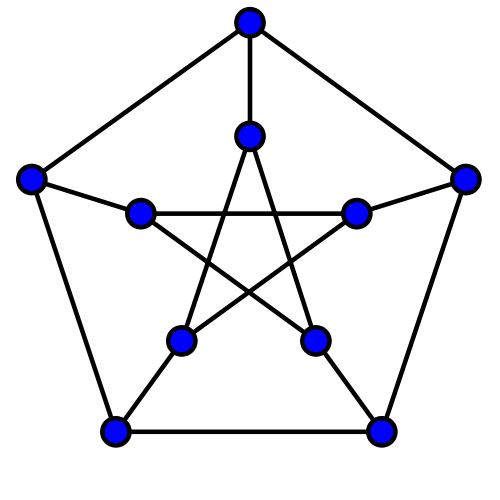}} 
\end{center}

\noindent
The symmetry group of the Petersen graph is \(\S_5\).  This graph has 10 vertices and 15 edges, so its Euler characteristic is \(-5\), which implies that its space of 1-cycles is 6-dimensional.  It thus gives a 6-dimensional crystal whose group of covering symmetries, an extension of \(\S_5\) by \(\Z^6\), acts transitively on arcs.  

Two more nice examples come from Klein's quartic curve, a Riemann surface of genus three on which the 336-element group \(\PGL(2,\mathbb{F}_7)\) acts as isometries.  These isometries preserve a tiling of Klein's quartic curve by 56 triangles, with 7 meeting at each vertex.   This picture is topologically correct, though not geometrically:

\begin{center} 
\raisebox{-0.0 em}{\includegraphics[scale = 0.35]{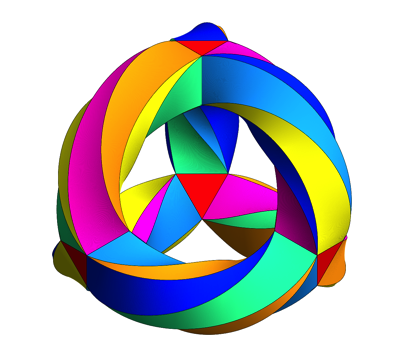}} 
\end{center}

\noindent
From this tiling we obtain a graph \(X\) embedded in Klein's quartic curve.  This graph has \(56 \times 3 / 2 = 84\) edges and \(56 \times 3 / 7 = 24\) vertices, so it has Euler characteristic \(-60\).   It thus gives a 61-dimensional topological crystal whose group of
covering symmetries, an extension of \(\PGL(2,\mathbb{F}_7)\) by \(\Z^{61}\), acts transitively on arcs.  There is also a dual tiling of Klein's curve by 24 heptagons, 3 meeting at each vertex.  This gives a graph with 84 edges and 56 vertices, hence Euler characteristic \(-28\).  From this we obtain a 29-dimensional topological crystal whose group of covering symmetries is an extension of \(\PGL(2,\mathbb{F}_7)\) by \(\Z^{29}\).  Again, this group acts transitively on arcs.

Here is some data about the examples discussed above.  Greg Egan used Mathematica to count the spanning trees.  In this table, \(\Z^k \, . \, G\)  denotes an extension of the group \(G\) by \(\Z^k\).

\vskip 2em
\begin{center}
\scalebox{0.9}{
\renewcommand{\arraystretch}{1.3}
\begin{tabular}[h]{|c|c|c|c|} \hline
 \bf Polyhedron &   \bf  Covering Symmetries & \bf Spanning Trees &  \bf Packing Fraction  \\ \hline
 tetrahedron & \(\Z^3 \, . \, \A_3\) & \(16\) & \(1/4\)   \\
\hline
cube  & \(\Z^5 \, . \, \B_3\) & 384 &  1/48 \\ 
\hline
octahedron & \(\Z^7 \, . \, \B_3\) & 384 & 1/64   \\
\hline
cuboctahedron & \(\Z^{13} \, . \, \B_3\)  & 331,776 &  1/27,648  \\
\hline
dodecahedron  & \(\Z^{11} \, . \, \H_3\)  &  5,184,000 & 1/259,200   \\
\hline
icosahedron  & \(\Z^{19} \, . \, \H_3\)  &  5,184,000 &   1/432,000 \\
\hline
icosidodecahedron  & \(\Z^{31} \, . \, \H_3\)  &   208,971,104,256,000  &  1/6,965,703,475,200 \\
\hline
\(n\)-gonal hosohedron  & \(\Z^{n-1} \, . \, (\S_n \times \Z/2) \) & \(n\) &  \(2/n\)  \\  
\hline
Petersen graph & \(\Z^6 \, . \, \S_5\) & 2,000 & 1/200   \\
\hline
Klein quartic  & \(\Z^{29} \, . \, \PGL(2,\F_7)\)  & 38,542,412,611,584,000,000 & 1/688,257,368,064,000,000   \\
heptagonal tiling & & & \\
\hline
Klein quartic  & \(\Z^{61} \, . \, \PGL(2,\F_7)\)  &    846,083,041,649,491,968
  & 1/35,253,460,068,728,832   \\
 triangular tiling & & & \\
 \hline 
\end{tabular}
}
\end{center}

\section{Conclusions}

We leave the reader with two questions.  First, when are all symmetries of the maximal abelian cover of \(X\) covering symmetries?  That is, when is \(\Cov(X) = \Aut(\overline{X})\)?  

Second, when is \(\Cov(X)\) actually a semidirect product of \(\Aut(X)\) and \(H_1(X,\Z)\)?
The short exact sequence  
\[       1 \longrightarrow H_1(X,\Z) \longrightarrow \Cov(X) \stackrel{\psi}{\longrightarrow} \Aut(X) \longrightarrow 1 \]
splits, making \(\Cov(X)\) into a semidirect product, whenever every automorphism of \(X\) fixes the basepoint \(x_0\).  It does not split when \(X\) is this graph:

\begin{center} 
\raisebox{-0.0 em}{\includegraphics[scale = 0.15]{grapheneGraph.png}} 
\end{center}

\noindent So, the short exact sequence does not describe \(\Cov(X)\) as a semidirect product in this case.  The reason is that while \(\Aut(X)\) is the symmetry group of a regular hexagon, this group does not act as symmetries fixing some atom in the corresponding crystal:

\vskip 2em
\begin{center} 
\raisebox{-0.0 em}{\includegraphics[scale = 0.2]{graphene2Ddark_small.png}} 
\end{center}

\noindent
Nonetheless, \(\Cov(X)\) can still be seen as a semidirect product of \(\Aut(X)\) and \(H_1(X,\Z)\) in this case, because this crystal has hexagonal symmetries about a point 
that is not an atom.  

On the other hand, suppose \(X\) is the graph giving the diamond crystal:

\begin{center} 
\raisebox{-0.0 em}{\includegraphics[scale = 0.15]{diamondGraph.png}} 
\end{center}

\noindent
In this case \(\Aut(X)\) is the symmetry group of a cube, and there is no way to get this group to act as symmetries of the diamond, so there is no way to express \(\Cov(X)\) as a semidirect product of \(\Aut(X)\) and \(H_1(X,\Z)\).     What general phenomenon is at work here?

\end{document}